\documentclass{amsart}

\usepackage{mathrsfs}
\usepackage{mathrsfs}
\usepackage{amsfonts}
\usepackage{amsfonts}
\usepackage{mathrsfs}
\usepackage{amsfonts}
\usepackage{amsfonts}
\usepackage{amsfonts}
\usepackage{mathrsfs}
\usepackage{amsfonts}
\usepackage{amssymb,amsmath}
\usepackage{amsthm}
\usepackage{amsfonts}

\newtheorem{theorem}{Theorem}[section]
\newtheorem{lemma}[theorem]{Lemma}

\theoremstyle{definition}
\newtheorem{definition}[theorem]{Definition}

\newtheorem{proposition}[theorem]{Proposition}

\theoremstyle{remark}
\newtheorem{remark}[theorem]{Remark}

\numberwithin{equation}{section}

\newcommand{\abs}[1]{\lvert#1\rvert}

\begin{document}

\title[Optimal Decay Viscoelastic]
 {Optimal Time Decay Rate for the Compressible Viscoelastic Equations in Critical Spaces}

%    Information for first author
\author[J.X.Jia]{Junxiong Jia}
%    Address of record for the research reported here
\address{Department of Mathematics,
Xi'an Jiaotong University,
 Xi'an
710049, China; Beijing Center for Mathematics and Information Interdisciplinary Sciences (BCMIIS);}
%    Current address
\email{jjx425@gmail.com}

\author[J. Peng]{Jigen Peng}
%    Address of record for the research reported here
\address{Department of Mathematics,
Xi'an Jiaotong University,
 Xi'an
710049, China; Beijing Center for Mathematics and Information Interdisciplinary Sciences (BCMIIS);}
%    Current address
\email{jgpeng@mail.xjtu.edu.cn}

%\author[K.X. Li]{Kexue Li}
%\address{Department of Mathematics,
%Xi'an Jiaotong University,
% Xi'an
%710049, China; Beijing Center for Mathematics and Information Interdisciplinary Sciences (BCMIIS);}
%%    Current address
%\email{kexueli@gmail.com}
%%    \thanks will become a 1st page footnote.
%\thanks{}

%   Information for second author
%\author[R.H. Pan]{Ronghua Pan}
%\address{Department of Mathematics,
%Georgia Institute of Technology,
%686 Cherry Street, Skiles Building
%Atlanta, GA 30332-0160 }
% \email{panrh@math.gatech.edu }

%    Information for third author
%\author[Z.D.Mei]{Zhan-Dong Mei}
%\address{Department of Mathematics,
%Xi'an Jiaotong University,
% Xi'an
%710049, China;}
% \email{mzhd1516@gmail.com}

%    General info
\subjclass[2010]{76N10, 35Q35, 35D35}

%%\date{January 1, 2001 and, in revised form, June 22, 2001.}

%%\dedicatory{This paper is dedicated to our advisors.}

\keywords{Compressible viscoelastic fluids, Critical space framework, Optimal time decay}

\begin{abstract}
In this paper, we are concerned with the convergence rates of the global strong solution to constant 
equilibrium state for the compressible viscoelastic fluids in the whole space. 
We combine both analysis about Green's matrix method and energy estimate method to get
optimal time decay rate in critical Besov space framework. Our result imply the optimal 
$L^{2}$-time decay rate and only need the initial data to be small in critical Besov space which have
very low regularity compared with traditional Sobolev space. 
\end{abstract}

\maketitle

%%%%%%%%%%%%%%%%%%%%%%%%%%%%%%%%%%%%%%%%%%%%%%%%%%%%%%%%%%%%%%%%%%%%%%%%%%%%%%%%%%%%%%%%%%%%%%%%%%%%%%%%%%%%%%%%%%%%%%%%%%%%%%%%%%%%%%%%%%%%%%%%%%%%%%%%%%%

\section{Introduction and main results}

Many fluids do not satisfy Newtonian law.
A viscoelastic fluid of the Oldroyd type is one of the classical non-Newtonian fluids which exhibits elastic behavior, such as memory effects.
The elastic properties of the fluid are described by associating the fluid motions with an energy functional of deformation tensor $U$.
Let us assume the elastic energy is $W(U)$, then the compressible viscoelastic system can be written as
\begin{align}\label{Viscoelastic 1}
\begin{split}
\begin{cases}
\partial_{t}\rho + \mathrm{div}(\rho u) = 0,    \\
\partial_{t}(\rho v) + \mathrm{div} (\rho v \otimes v) + \nabla P(\rho) = \mathrm{div}(2\mu \mathcal{D}(v) + \nabla (\lambda \mathrm{div}(v)))   \\
\quad\quad + \mathrm{div}\left( \frac{W_{U}(U)U^{T}}{\mathrm{det}(U)} \right), \\
\partial_{t}U + u\cdot \nabla U = \nabla u U.
\end{cases}
\end{split}
\end{align}
Here $\rho$ is the density and $v(x, t)$ is the velocity of the fluid.
The pressure $P(\rho)$ is a given state equation with $P'(\rho) > 0$ for any $\rho$ and
$\mathcal{D}(v) = \frac{1}{2}(\nabla v + \nabla v^{T})$ is the strain tensor. The Lam\'{e} coefficient
$\mu$ and $\lambda$ are assumed to satisfy
\begin{align}
\mu > 0 \quad \mathrm{and} \quad \lambda+2\mu > 0.
\end{align}
Such a condition ensures ellipticity for the operator $-\nabla (2\mu \mathcal{D}\cdot) - \nabla (\lambda \nabla \cdot)$ and
is satisfied in the physical case, where $\lambda + 2\mu / N \approx 0$. Moreover, $W_{U}(U)$ is the Piola-Kirchhoff tensor
and $\frac{W_{U}(U)U^{T}}{\mathrm{det}(U)}$ the Cauchy-Green tensor, respectively. For a special case of the Hookean
linear elasticity, $W(U) = |U|^{2}$.

For the incompressible viscoelastic fluids, there are many important works recently. In \cite{leizhen2008}, the author proved the well posedness problem
and find the relation
\begin{align*}
\nabla_{k}F^{ij} - \nabla_{j}F^{ik} = F^{lj}\nabla_{l}F^{ik} - F^{lk}\nabla_{l}F^{ij},
\end{align*}
with $F = U - I$. This relation indicates that the linear term $\nabla\times F$ is actually a higher order term.
F. Lin, C. Liu and P. Zhang \cite{Linfanghua2005,Linfanghua2008} proved the local well posedness in Hilbert space $H^{s}$, and global well posedness
with small initial data. In the proof of the global part, they capture the damping mechanism on $F$ through very subtle energy estimates.
At last, in \cite{LpIncompressibleViscoelastic}, the author proved the global well posedness of the incompressible version of system (\ref{Viscoelastic 1})
in the critical $L^{p}$ framework which allows us to construct the unique global solution for highly oscillating initial velocity.

For compressible viscoelastic fluids, in \cite{L2viscoelastic,XianpengHu} the authors proved
the local and global well-posedness in the $L^{2}$ based critical Besov type space. Their work used the properties
of the viscoelastic fluids deeply and their results indicated that the deformation tensor $U$ plays a similar role as the density $\rho$.
It should be mentioned that the global existence of a smooth solutions is still an open problem, even in for incompressible viscoelastic fluids.
P. Lions and N. Masmoudi \cite{Lions2000} proved the global existence of a weak solution with general initial data in the case that the contribution of the strain rate in the constitutive equation is neglected.
Recently, in \cite{Jia2014}, the author proved the global well-posedness in $L^{p}$ based critical Besov space.

Besides well-posedness theory, optimal time decay rate problem is another important subject.
There are many papers studied optimal time decay rate for compressible Navier-Stokes system \cite{ponce1985,yanguo2012,HailiangLi2009,hoff1995,hoff1997,liutaiping1998,DLLi,Okita2014}.
However, due to the complexity of the compressible viscoelastic equations, there are little results about viscoelastic system.
Recently, X. Hu and G. Wu in \cite{decayhu} give a detailed analysis about time-decay rate in the Sobolev space framework.
They split the whole system into two small systems and then the analysis becomes possible.
In \cite{Jia2014}, the author used estimates in homogeneous space and negative Besov space to
give a slow decay rate when the initial data only small in Besov space with low regularity.
The main goal of this paper is to get the optimal time decay rate when the initial data just small in critical Besov space framework.
Hence, we can link the results in \cite{decayhu} and \cite{Jia2014} to give a more elaborate
characterization about time decay rate for compressible viscoelastic system.

In paper \cite{L2viscoelastic,XianpengHu}, they proved the following proposition which reveal some intrinsic properties about
compressible viscoelastic equations.
\begin{proposition}\label{specialpro}
The density $\rho$ and deformation tensor $U$ in (\ref{Viscoelastic 1}) satisfy the following relations:
\begin{align}\label{special1}
\begin{split}
& \mathrm{div}\left( \frac{U^{T}}{\mathrm{det}U} \right) = 0, \quad \mathrm{div}(\rho U^{T}) = 0, \quad \rho \mathrm{det}U = 1,   \\
& \text{and} \quad U^{lk} \nabla_{l}U^{ij} - U^{lj}\nabla_{l}U^{ik} = 0,
\end{split}
\end{align}
if the initial data $(\rho, U)|_{t = 0} = (\rho_{0}, U_{0})$ satisfies
\begin{align}\label{special2}
\begin{split}
& \mathrm{div}\left( \frac{U_{0}^{T}}{\mathrm{det}U_{0}} \right) = 0, \quad \mathrm{div}(\rho_{0} U_{0}^{T}) = 0, \quad \rho_{0} \mathrm{det}U_{0} = 1,   \\
& \text{and} \quad U_{0}^{lk} \nabla_{l}U_{0}^{ij} - U_{0}^{lj}\nabla_{l}U_{0}^{ik} = 0,
\end{split}
\end{align}
respectively.
\end{proposition}
Using Proposition \ref{specialpro}, the last term in the second equation of (\ref{Viscoelastic 1}) can be rewritten as
\begin{align}\label{transform}
\nabla_{j} \left( \frac{\frac{\partial W(U)}{\partial U^{ik}}U^{jk}}{\mathrm{det}U} \right)
= \frac{1}{\mathrm{det}U} U^{jk} \nabla_{j}\left( \frac{\partial W(U)}{\partial U^{ik}} \right)
= \rho U^{jk} \nabla_{j} \left( \frac{\partial W(U)}{\partial U^{ik}} \right).
\end{align}
As in \cite{XianpengHu}, without loss of generality, we consider Hookean linear elasticity, $W(U) = |U|^{2}$ in the following part of this paper.
Note that this does not reduce the essential difficulties. All results can be easily generalized to the case of more general elastic energy
functionals. In view of (\ref{transform}), we will consider the following system
\begin{align}\label{Viscoelastic 2}
\begin{split}
\begin{cases}
\partial_{t}\rho + \mathrm{div}(\rho v) = 0,      \\
\rho \partial_{t}v^{i} + \rho v\cdot\nabla v^{i} - \mathrm{div}(2\mu \mathcal{D}(v)) - \nabla (\lambda \mathrm{div}v) + \nabla P(\rho) = \rho U^{jk} \nabla_{j} U^{ik},       \\
\partial_{t} U + v\cdot\nabla U = \nabla v U,     \\
(\rho, v, U)|_{t=0} = (\rho_{0}, v_{0}, U_{0})
\end{cases}
\end{split}
\end{align}
where the initial data satisfies (\ref{special2}).

We now state our main result of this paper which gives the optimal $L^{2}$-time decay rate for strong solutions in critical Besov spaces.
\begin{theorem}\label{MainTheorem}
Assume that dimension $n = 3$, $\bar{\rho}$ be a constant and $I$ stands for the identity vector $(1,1,1)$.
There exists $\delta > 0$ such that if $v_{0} \in B_{2,1}^{n/2-1} \cap \dot{B}_{1,\infty}^{0}$,
$\rho_{0} - \bar{\rho} \in B_{2,1}^{n/2} \cap \dot{B}_{1,\infty}^{0}$,
$U - I \in B_{2,1}^{n/2} \cap \dot{B}_{1,\infty}^{0}$ and
\begin{align*}
\|(\rho_{0} - \bar{\rho}, U - I)\|_{B_{2,1}^{n/2} \cap \dot{B}_{1,\infty}^{0}} + \|v_{0}\|_{B_{2,1}^{n/2-1} \cap \dot{B}_{1,\infty}^{0}} \leq \delta,
\end{align*}
then problem (\ref{Viscoelastic 2}) has a unique global solution
$(\rho - \bar{\rho}, v, U - I) \in C(\mathbb{R}^{+}; B_{2,1}^{n/2}) \times \left( C(\mathbb{R}^{+}; B_{2,1}^{n/2-1})
\cap L^{1}(\mathbb{R}^{+}; \dot{B}_{2,1}^{n/2+1}) \right)^{n} \times \left( C(\mathbb{R}^{+}; B_{2,1}^{n/2}) \right)^{n\times n}$.
Furthermore, there exists constant $C_{0} > 0$, and we have
\begin{align}\label{DecayRate}
\|(\rho - \bar{\rho}, v, U - I)(t)\|_{B_{2,1}^{n/2-1}} \leq C_{0} (1+t)^{-n/4},
\end{align}
for $t \geq 0$.
\end{theorem}

\begin{remark}\label{MainRemark_1}
From \cite{decayhu}, we know that the optimal $L^{2}$-time decay rate for compressible viscoelastic equations is
\begin{align}\label{OptimalTimeDecayTradictional}
\|(\rho - \bar{\rho}, v, U-I)(t)\|_{L^{2}} \leq C (1+t)^{-n/4}.
\end{align}
Due to $B_{2,1}^{n/2-1} \subset L^{2}$, the convergence rate of (\ref{DecayRate}) is optimal.
\end{remark}

To prove Theorem \ref{MainTheorem}, we split the system by Littlewood-Paley operator to
low frequency part and high frequency part. For the low frequency part, we decompose the system into
three small system and analyze the green's matrix carefully as in \cite{Jia2014, decayhu} for each small system.
Due to the fine properties of homogeneous space and singular operators, we can then combine the estimates
for small systems together to finally obtain an estimates about the whole system.
For the high frequency part, we reformulate the system as in \cite{L2viscoelastic} and using energy estimates
in Besov space framework to get an appropriate a prior estimates.

The paper is organized as follows.
In Section 2, we introduce the notation, some properties of Besov space and some important Lemmas.
In Section 3, we split the system into three small system and give the estimates for low frequency part.
In Section 4, we transform the system into an equivalent form and prove an estimates for high frequency part.
In Section 5, we give the proof of Theorem \ref{MainTheorem}.

\section{Preliminaries}

In this section we first introduce the notation which will be used throughout this paper.
Secondly, we give some basic knowledge about Besov space.
At last, we present some useful Lemmas and Theorems.

\subsection{Notation}

Let $n$ stands for the dimension, $L^{p} (1\leq p \leq \infty)$ denote the usual $L^{p}$-Lebesgue space on $\mathbb{R}^{n}$.
$[z]$ stands for the integer part of a number $z \in \mathbb{R}$.
The inner-product of $L^{2}$ is denoted by $(\cdot , \cdot)$.
If $S$ is any nonempty set, sequence space $\ell^{p}(S)$ denotes the usual $\ell^{p}$ sequence space on $S$.
For any integer $\ell \geq 0$, $\nabla^{\ell}f$ denotes all of $\ell$-th derivatives of $f$.

For a function $f$, we denote its Fourier transform by $\mathcal{F}[f] = \hat{f}$:
\begin{align*}
\mathcal{F}[f](\xi) = \hat{f}(\xi) = (2\pi)^{-n/2} \int_{\mathbb{R}^{n}} f(x) e^{-ix\cdot\xi} dx.
\end{align*}
The inverse of $\mathcal{F}$ is denoted by $\mathcal{F}^{-1}[f] = \check{f}$:
\begin{align*}
\mathcal{F}^{-1}[f](x) = \check{f}(x) = (2 \pi)^{-n/2} \int_{\mathbb{R}^{n}} f(\xi) e^{i\xi \cdot x} d\xi.
\end{align*}

\subsection{Besov spaces}

In this section, we will give some basic knowledge about Besov space, which can be found in \cite{FourierPDE}.
First we introduce the dyadic partition of unity. We can use for instance any $(\phi, \chi) \in C^{\infty}$,
such that $\phi$ is supported in $\{\xi \in \mathbb{R}^{n} : 3/4 \leq \abs{\xi} \leq 8/3 \}$,
$\chi$ is supported in $\{ \xi \in \mathbb{R}^{n} : |\xi| \leq 4/3 \}$ such that
\begin{align*}
& \chi(\xi) + \sum_{q \geq 0} \phi(2^{-q}\xi) = 1 \quad \xi \in \mathbb{R}^{n}, \\
& \sum_{q \in \mathbb{Z}} \phi(2^{-q} \xi) = 1 \quad \mathrm{if} \quad \xi \neq 0.
\end{align*}
Denoting $h = \mathcal{F}^{-1}[\phi]$ and $\tilde{h} = \mathcal{F}^{-1}\chi$,
we define the dyadic blocks as follows
\begin{align*}
& \Delta_{-1}u = \chi(D)u = \tilde{h} * u, \\
& \Delta_{q}u = \phi(2^{-q}D)u = 2^{qn}\int_{\mathbb{R}^{n}} h(2^{q}y)u(x-y) dy \quad \text{if} \quad q \geq 0, \\
& \dot{\Delta}_{q}u = \phi(2^{-q}D)u = 2^{qn}\int_{\mathbb{R}^{n}} h(2^{q}y) u(x-y) dy \quad \text{if} \quad q \in \mathbb{Z}.
\end{align*}
The low-frequency cut-off operator is defined by
\begin{align*}
S_{q}u = \sum_{-1 \leq k \leq q-1}\Delta_{q}u, \quad \dot{S}_{q}u = \sum_{k\leq q-1} \dot{\Delta}_{k}u.
\end{align*}
The formal two decompositions
\begin{align*}
u = \sum_{q \geq -1} \Delta_{q}u , \quad u = \sum_{q\in \mathbb{Z}} \dot{\Delta}_{q} u
\end{align*}
are called inhomogeneous and homogeneous Littlewood-Paley decomposition respectively.

Let us give the definition of inhomogeneous Besov space as follows.
\begin{definition}
For $s \in \mathbb{R}$ and $1\leq p, r\leq \infty$, and $u \in \mathcal{S}'$. The inhomogeneous Besov space
$B_{p,r}^{s}$ consists of distributions $u$ in $\mathcal{S}'$ such that
\begin{align*}
\|u\|_{B_{p,r}^{s}} := \left( \sum_{q \geq -1} 2^{rjs} \|\Delta_{q}u\|_{L^{p}}^{r} \right)^{1/r} < +\infty.
\end{align*}

\end{definition}

Let us now introduce the homogeneous Besov space.
\begin{definition}
We denote by $\mathcal{S}_{h}'$ the space of tempered distributions $u$ such that
\begin{align*}
\lim_{q \rightarrow -\infty} S_{q}u = 0 \quad \mathrm{in} \quad \mathcal{S}'.
\end{align*}
\end{definition}

\begin{definition}
Let $s$ be a real number and $(p,r)$ be in $[1, \infty]^{2}$. The homogeneous Besov space $\dot{B}_{p,r}^{s}$ consists of distributions
$u$ in $\mathcal{S}_{h}'$ such that
\begin{align*}
\|u\|_{\dot{B}_{p,r}^{s}} := \left( \sum_{q \in \mathbb{Z}} 2^{rjs} \|\dot{\Delta}_{q}u\|_{L^{p}}^{r} \right)^{1/r} < +\infty.
\end{align*}
\end{definition}
From now on, the notation $\dot{B}_{p}^{s}$, $B_{p}^{s}$ will stand for $\dot{B}_{p,1}^{s}$ and $B_{p,1}^{s}$ respectively.
The notation $\dot{B}^{s}$, $B^{s}$ will stand for $\dot{B}_{2,1}^{s}$ and $B_{2,1}^{s}$ respectively.

The study of non stationary PDE's usually requires spaces of type $L_{T}^{r}(X) := L^{r}(0,T; X)$ for appropriate Banach spaces $X$.
In our case, we expect $X$ to be a Besov spaces, so that it is natural to localize the equations through Littlewood-Paley decomposition.
We then get estimates for each dyadic block and perform integration in time. However, in doing so, we obtain bounds in spaces which are not
of type $L^{r}(0,T; B_{p}^{s})$ or $L^{r}(0,T; \dot{B}_{p}^{s})$. This approach was initiated in \cite{chemin}
naturally leads to the following definitions for the inhomogeneous Besov space.
\begin{definition}
Let $(r,p) \in [1, +\infty]^{2}$, $T \in (0, +\infty]$ and $s\in \mathbb{R}$. We set
\begin{align*}
\|u\|_{\tilde{L}_{T}^{r}(B_{p}^{s})} := \sum_{q \in \mathbb{Z}} 2^{qs} \left( \int_{0}^{T} \|\Delta_{q}u(t)\|_{L^{p}}^{r} \, dt \right)^{1/r}
\end{align*}
and
\begin{align*}
\tilde{L}_{T}^{r}(B_{p}^{s}) := \left\{ u \in L_{T}^{r}(B_{p}^{s}), \|u\|_{\tilde{L}_{T}^{r}(B_{p}^{s})}< +\infty \right\}.
\end{align*}
\end{definition}
Owing to Minkowski inequality, we have $\tilde{L}_{T}^{r}(B_{p}^{s}) \hookrightarrow L_{T}^{r}(B_{p}^{s})$. That embedding is strict in general
if $r>1$. We will denote by $\tilde{C}_{T}(B_{p}^{s})$ the set of function $u$ belonging to
$\tilde{L}_{T}^{\infty}(B_{p}^{s})\cap C([0,T]; B_{p}^{s})$.
For the homogeneous Besov space, we can define similarly.

Let $X$ stands for $B$ or $\dot{B}$, we have the following interpolation inequality:
\begin{align*}
\|u\|_{\tilde{L}_{T}^{r}(X_{p}^{s})} \leq \|u\|_{\tilde{L}_{T}^{r_{1}}(X_{p}^{s_{1}})}^{\theta} \|u\|_{\tilde{L}_{T}^{r_{2}}(X_{p}^{s_{2}})}^{1-\theta},
\end{align*}
with
\begin{align*}
\frac{1}{r} = \frac{\theta}{r_{1}} + \frac{1-\theta}{r_{2}} \quad \mathrm{and} \quad s = \theta s_{1} + (1-\theta) s_{2},
\end{align*}
and the following embeddings
\begin{align*}
\tilde{L}_{T}^{r}(X_{p}^{n/p}) \hookrightarrow L_{T}^{r}(\mathcal{C}_{0}) \quad \mathrm{and} \quad \tilde{C}_{T}(X_{p}^{n/p})
\hookrightarrow C([0,T]\times \mathbb{R}^{n}).
\end{align*}

Another important space is the bybrid Besov space, we give the definitions and collect some properties.
\begin{definition}
let $s,t \in \mathbb{R}$. We set
\begin{align*}
\|u\|_{B_{q,p}^{s,t}} := \sum_{q \leq R_{0}} 2^{qs} \|\dot{\Delta}_{q}u\|_{L^{q}} + \sum_{q > R_{0}}2^{qt} \|\dot{\Delta}_{q}u\|_{L^{p}}.
\end{align*}
and
\begin{align*}
B_{q,p}^{s,t}(\mathbb{R}^{N}) := \left\{ u \in \mathcal{S}_{h}'(\mathbb{R}^{N}) : \|u\|_{B_{q,p}^{s,t}} < +\infty \right\},
\end{align*}
where $R_{0}$ is a fixed large enough number determined in the proof of global existence.
\end{definition}

\begin{lemma}

1) We have $B_{2, 2}^{s,s} = \dot{B}^{s}$.

2) If $s\leq t$ then $B_{p,p}^{s,t} = \dot{B}_{p}^{s} \cap \dot{B}_{p}^{t}$. Otherwise, $B_{p,p}^{s,t} = \dot{B}_{p}^{s} + \dot{B}_{p}^{t}$.

3) The space $B_{p,p}^{0,s}$ coincide with the usual inhomogeneous Besov space.

4) If $s_{1} \leq s_{2}$ and $t_{1} \geq t_{2}$ then $B_{p,p}^{s_{1},t_{1}} \hookrightarrow B_{p,p}^{s_{2},t_{2}}$.

5) Interpolation: For $s_{1}, s_{2}, \sigma_{1}, \sigma_{2} \in \mathbb{R}$ and $\theta \in [0,1]$, we have
\begin{align*}
\|f\|_{B_{2,p}^{\theta s_{1} + (1-\theta)s_{2}, \theta \sigma_{1} + (1-\theta)\sigma_{2}}} \leq \|f\|_{B_{2,p}^{s_{1},\sigma_{1}}}^{\theta}
\|f\|_{B_{2,p}^{s_{2},\sigma_{2}}}^{1-\theta}.
\end{align*}
\end{lemma}
From now on, the notation $B_{p}^{s,t}$ will stand for $B_{p,p}^{s,t}$ and the notation $B^{s,t}$ will stand for $B_{2,2}^{s,t}$.
For more information about Besov space and hybrid Besov space, we give reference \cite{FourierPDE,Danchin2005,NavierStokesLp,NavierStokesLpDanchin}.

In the last of this introduction, for the reader's convenience, we list an important Lemma \cite{NavierStokesLp, L2viscoelastic} which will be used in the following.
\begin{lemma}\label{yinli}
Let $F$ be a homogeneous smooth function of degree $m$. Suppose
$1 - n/2 < \rho \leq 1 + n/2$ and $-1/n < \rho' \leq n/2+1$.
Then the following inequialities hold:
\begin{align*}
& |(F(D)\Delta_{q}(v\cdot\nabla c) | F(D)\Delta_{q}c)|  \\
& \quad
\leq \, C\alpha_{q}2^{-q(\rho'-m)}\|v\|_{\dot{B}^{n/2+1}}\|c\|_{\dot{B}^{\rho'}}
\|F(D)\Delta_{q}c\|_{L^{2}}, 	\\
& |(F(D)\Delta_{q}v\cdot\nabla c | F(D)\Delta_{q}c)|    \\
& \quad
\leq \, C\alpha_{q}2^{-q(\rho - m)}\min(2^{q},1)\|v\|_{\dot{B}^{n/2+1}}
\|c\|_{B^{\rho-1,\rho}}\|F(D)\Delta_{q}c\|_{L^{2}}, 	\\
& |(F(D)\Delta_{q}(v\cdot\nabla c) | \Delta_{q}d)| + |(\Delta_{q}(v\cdot d) | F(D)\Delta_{q}c)| 	\\
& \quad
\leq \, C\alpha_{q}2^{-q(\rho - m)}\min(2^{q},1)\|v\|_{\dot{B}^{n/2+1}}
(\|c\|_{B^{\rho-1,\rho}}\|\Delta_{q}d\|_{L^{2}}+\|d\|_{B^{\rho-1,\rho}}\|\Delta_{q}c\|_{L^{2}}), \\
& |(F(D)\Delta_{q}(v\cdot c) | \Delta_{q}d)| + |(\Delta_{q}(v\cdot \nabla d)|
F(D)\Delta_{q}c)|	\\
& \quad
\leq \, C\alpha_{q}\|v\|_{\dot{B}^{n/2+1}}
(2^{-q\rho'}\|F(D)\Delta_{q}c\|_{L^{2}}\|d\|_{\dot{B}^{\rho'}}  \\
& \quad\quad\,\,
+ 2^{-q(\rho-m)}\min(2^{q},1)\|d\|_{B^{\rho-1, \rho}}\|\Delta_{q}d\|_{L^{2}} ).
\end{align*}
\end{lemma}

\subsection{Useful Theorems}
In the part, we will list two Theorems about well-posedness of equations (\ref{Viscoelastic 2}) which are
essential for our proof of Theorem \ref{MainTheorem}.
Denote
\begin{align*}
E_{p}(T) := \{ v \in C([0,T]; B^{n/p}_{p}), \, \partial_{t}v, \nabla^{2}v \in L^{1}(0, T; B^{n/p}_{p}) \}.
\end{align*}
For $v \in E_{p}(T)$ will be endowed with the norm
\begin{align*}
\|v\|_{E_{p}(T)} := \|v\|_{L_{T}^{\infty}(B^{n/p-1}_{p})} + \|\partial_{t}v, \nabla^{2}v\|_{L_{T}^{1}(B^{n/p-1}_{p})}.
\end{align*}
Through similar methods used in \cite{Jia2014} or just change the low frequency estimates in \cite{L2viscoelastic},
we will obtain Theorem \ref{localeuler}. Due to the proof has no new ingredients, we omit it.
\begin{theorem}\label{localeuler}
Let $1 < p < 2n$  and $n \geq 2$. Let $v_{0}$ be vector field in $B^{n/p-1}_{p}$. Assume that $\rho_{0}$ satisfies
$a_{0} := \rho_{0} - 1 \in B^{n/p}_{p}$ and $U_{0}$ satisfies $F_{0} := U_{0} - I \in B^{n/p}_{p}$ and
\begin{align}
\inf_{x} \rho_{0}(x) > 0.
\end{align}
Then system (\ref{Viscoelastic 2}) has a unique local solution $(\rho, v, U)$ with
$v \in E_{p}(T)$, $U-I \in C([0,T]; B^{n/p}_{p})$, $\rho$ bounded away from $0$ and $\rho - 1 \in C([0,T]; B^{n/p}_{p})$.
\end{theorem}

Denote:
\begin{align*}
\mathcal{E}^{s} := \Big\{ (a, u, F) \in & (L^{1}(0,\infty; B_{2,p}^{s_{p}+1,s}) \cap \tilde{L}^{\infty}(0,\infty; B_{2,p}^{s_{p}-1,s}) ) \\
& \times ( L^{1}(0,\infty; B_{2,p}^{s_{p}+1,s+1}) \cap \tilde{L}^{\infty}(0,\infty; B_{2,p}^{s_{p}-1,s-1}) )^{n}    \\
& \times (L^{1}(0,\infty; B_{2,p}^{s_{p}+1,s}) \cap \tilde{L}^{\infty}(0,\infty; B_{2,p}^{s_{p}-1,s}))^{n\times n}\Big\},
\end{align*}
where $s_{p} = s - \frac{n}{p} + \frac{n}{2}$.

The global well posedness of equations (\ref{Viscoelastic 2}) in Besov space framework are as follows.
\begin{theorem}\cite{Jia2014}\label{global}
Let $\bar{\rho} > 0$ be a constant such that $P'(\bar{\rho}) > 0$. Suppose that $n=3$. There exist two positive constants $\alpha_{0}$ and $C$
such that for all $(\rho_{0}, v_{0}, U_{0})$ with $\rho_{0} - \bar{\rho} \in B_{2,p}^{n/2-1,n/p}$, $U_{0} - I \in B_{2,p}^{n/2-1,n/p}$,
$v_{0} \in B_{2,p}^{n/2-1,n/p-1}$, and
\begin{align}\label{AssumptionInGlobal}
\|\rho_{0} - \bar{\rho}\|_{B_{2,p}^{n/2-1,n/p}} + \|v_{0}\|_{B_{2,p}^{n/2-1,n/p-1}} + \|U_{0}-I\|_{B_{2,p}^{n/2-1,n/p}} \leq \alpha_{0},
\end{align}
then if $2 \leq p < n$, system (\ref{Viscoelastic 2}) has a unique global solution $(\rho - \bar{\rho}, v, U-I) \in \mathcal{E}^{n/p}$ with
\begin{align*}
\|(\rho-\bar{\rho}, v ,U-I)\|_{\mathcal{E}^{n/p}} \leq C \Big( & \|\rho_{0} - \bar{\rho}\|_{B_{2,p}^{n/2-1,n/p}} + \|v_{0}\|_{B_{2,p}^{n/2-1,n/p-1}}    \\
& + \|U_{0}-\bar{U}\|_{B_{2,p}^{n/2-1,n/p}} \Big).
\end{align*}
\end{theorem}

\begin{remark}\label{PriorAssumption}
Taking $p = 2$ in Theorem \ref{global}, we will get global well-posedness in the critical homogeneous Besov space framework.
Assume
\begin{align*}
\|(\rho_{0} - \bar{\rho}, U - I)\|_{B^{n/2} \cap \dot{B}_{1,\infty}^{0}} + \|v_{0}\|_{B^{n/2-1} \cap \dot{B}_{1,\infty}^{0}} \leq \delta,
\end{align*}
as in our main Theorem \ref{MainTheorem}.
If $\delta > 0$ is taken to be small enough, the above assumption will imply (\ref{AssumptionInGlobal}),
hence, we obtain the results in Theorem \ref{global}.
Particularly, we know
\begin{align}\label{AssumptionForTimeIntegralOfV}
\int_{0}^{\infty} \|v\|_{\dot{B}^{n/2+1}} dt \leq C \delta.
\end{align}
This estimate plays an essential role when we estimate the nonlinear terms.
\end{remark}

\section{Analysis about low frequency part}

In this section, we first decompose the system into three small scale system, then analyze the semigroup carefully for the low frequency.
Without loss of generality, we assume $P'(1) = 1$, $\bar{\rho} = 1$ and set $\nu = \lambda + 2\mu$, $\mathcal{A} = \mu \Delta + (\lambda + \mu)\nabla\mathrm{div}$.
Define
\begin{align*}
& \quad\quad\quad\quad\quad K(a) = \frac{P'(1+a)}{1+a} - 1,   \quad  d = \Lambda^{-1}\mathrm{div}v,       \\
& \quad\quad\quad\quad\quad \Omega = \Lambda^{-1}\mathrm{curl}v \quad \text{with } (\mathrm{curl}v)_{ij} = \partial_{x_{j}}v^{i} - \partial_{x_{i}}v^{j},    \\
& \quad\quad\quad\quad\quad\quad\quad \mathcal{E}_{ij} = \Lambda^{-1}\partial_{x_{i}}\Lambda^{-1}\partial_{x_{j}}(F^{ij}+F^{ji}),       \\
& \mathcal{W} = \Lambda^{-1}\partial_{x_{k}}(F^{lk}\nabla_{l}F^{ij}-F^{lj}\nabla_{l}F^{ik}) - \Lambda^{-1}\partial_{x_{k}}
(F^{lk}\nabla_{l}F^{ji}-F^{li}\nabla_{l}F^{jk}).
\end{align*}
where
\begin{align*}
\Lambda^{s}f = \mathcal{F}^{-1}(|\xi|^{s}\hat{f}) \quad \text{for  } s \in \mathbb{R}.
\end{align*}
Performing same procedure as in \cite{XianpengHu}, we will obtain
\begin{align}\label{newform}
\begin{split}
\begin{cases}
\partial_{t}a + \Lambda d = L - v\cdot\nabla a,  \\
\partial_{t}d - \mu \Delta d - 2\Lambda a = G - v\cdot\nabla d,  \\
\partial_{t}\mathcal{E} + 2 \Lambda d = J - v\cdot\nabla \mathcal{E},   \\
\partial_{t}(F^{T} - F) + \Lambda \Omega = I - v\cdot\nabla (F^{T} - F),   \\
\partial_{t}\Omega - \mu \Delta \Omega - \Lambda (F^{T} - F) = H - v\cdot\nabla \Omega,
\end{cases}
\end{split}
\end{align}
where the equation about $d$ have the following equivalent form
\begin{align}\label{equi}
\partial_{t}d - \nu \Delta d - \Lambda \mathcal{E} = K - v\cdot\nabla d,
\end{align}
where
\begin{align*}
& L = - a\mathrm{div}v,    \\
& G = v\cdot\nabla d + \Lambda^{-1}\mathrm{div}(-v\cdot\nabla v + F\nabla F - K(a)\nabla a - \frac{a}{1+a}\mathcal{A}v - \mathrm{div}(aF)),  \\
& H = v\cdot\nabla \Omega + \Lambda^{-1}\mathrm{curl}(-v\cdot\nabla v + F\nabla F - K(a)\nabla a - \frac{a}{1+a}\mathcal{A}v) + \mathcal{W},      \\
& I = (\nabla v F)^{T} - \nabla v F,     \\
& J = - [\Lambda^{-1}\partial_{x_{i}}\Lambda^{-1}\partial_{x_{j}}, v^{k}]\partial_{x_{k}}(F^{ij} + F^{ji})      \\
& \quad\quad + \Lambda^{-1}\partial_{x_{i}}\Lambda^{-1}\partial_{x_{j}}((\nabla v F)^{ij} + (\nabla v F)^{ji}),      \\
& K = v\cdot\nabla d + \Lambda^{-1}\mathrm{div}(-v\cdot\nabla v + F\nabla F - K(a)\nabla a - \frac{a}{1+a}\mathcal{A}v + \mathrm{div}(aF)).
\end{align*}
Here, we denote
\begin{align*}
& M_{1}(t) := \sup_{0 \leq \tau \leq t} (1+\tau)^{n/4}\left( \|a(\tau)\|_{B^{n/2-1,n/2}} + \|d(\tau)\|_{\dot{B}^{n/2-1}} \right),    \\
& M_{2}(t) := \sup_{0 \leq \tau \leq t} (1+\tau)^{n/4}\left( \|\mathcal{E}(\tau)\|_{B^{n/2-1,n/2}} + \|d(\tau)\|_{\dot{B}^{n/2-1}} \right), \\
& M_{3}(t) := \sup_{0 \leq \tau \leq t} (1+\tau)^{n/4}\left( \|(F^{T} - F)(\tau)\|_{B^{n/2-1,n/2}} + \|\Omega(\tau)\|_{\dot{B}^{n/2-1}} \right),    \\
& M_{4}(t) := \sup_{0 \leq \tau \leq t} (1+\tau)^{n/4}\left( \|a\|_{L^{2}} + \|F\|_{L^{2}} + \|v\|_{L^{2}} \right),
\end{align*}
and
\begin{align*}
M(t) := \sup_{0 \leq \tau \leq t} (1+\tau)^{n/4}\left( \|a\|_{B^{n/2}} + \|F\|_{B^{n/2}} + \|v\|_{B^{n/2-1}} \right).
\end{align*}
From the basic properties of Besov space, we easily know
\begin{align*}
M(t) \approx M_{1}(t) + M_{2}(t) + M_{3}(t) + M_{4}(t),
\end{align*}
under the smallness condition of initial data, where we used (5.13) to (5.15) in \cite{Jia2014}.
We here also denote
\begin{align}\label{DefineF1F2}
M_{1} = \left(
         \begin{array}{c}
           L - v\cdot\nabla a \\
           G - v\cdot\nabla s \\
         \end{array}
       \right),
\quad
M_{2} = \left(
       \begin{array}{c}
         J - v\cdot\nabla \mathcal{E} \\
         K - v\cdot\nabla d \\
       \end{array}
     \right),
\end{align}
and
\begin{align}\label{DefineF3}
M_{3} = \left(
       \begin{array}{c}
         I - v\cdot\nabla (F^{T} - F) \\
         H - v\cdot\nabla \Omega \\
       \end{array}
     \right).
\end{align}

Now, we need to introduce the following linearized system with convection terms.
\begin{align}\label{linear}
\begin{split}
\begin{cases}
\partial_{t}a + \Lambda d = L - v\cdot\nabla a,  \\
\partial_{t}d - \mu \Delta d - 2\Lambda a = G - v\cdot\nabla d,  \\
\partial_{t}\mathcal{E} + 2\Lambda d = J - v\cdot\nabla \mathcal{E},   \\
\partial_{t}(F^{T} - F) + \Lambda \Omega = I - v\cdot\nabla (F^{T} - F),   \\
\partial_{t}\Omega - \mu \Delta \Omega - \Lambda (F^{T} - F) = H - v\cdot\nabla \Omega.
\end{cases}
\end{split}
\end{align}
We can decompose the above system into three subsystems.
\begin{align}\label{linear1}
\begin{split}
\begin{cases}
\partial_{t}a + \Lambda d = L - v\cdot\nabla a,  \\
\partial_{t}d - \mu \Delta d - 2\Lambda a = G - v\cdot\nabla d.
\end{cases}
\end{split}
\end{align}
\begin{align}\label{linear2}
\begin{split}
\begin{cases}
\partial_{t}\mathcal{E} + 2 \Lambda d = J - v\cdot\nabla \mathcal{E},   \\
\partial_{t}d - \nu \Delta d - \Lambda \mathcal{E} = K - v\cdot\nabla d.
\end{cases}
\end{split}
\end{align}
\begin{align}\label{linear3}
\begin{split}
\begin{cases}
\partial_{t}(F^{T} - F) + \Lambda \Omega = I - v\cdot\nabla (F^{T} - F),   \\
\partial_{t}\Omega - \mu \Delta \Omega - \Lambda (F^{T} - F) = H - v\cdot\nabla\Omega.
\end{cases}
\end{split}
\end{align}

It is easily observed that the above three system are similar, so we now only study the following linear system.
\begin{align}\label{summary}
\begin{split}
\begin{cases}
\partial_{t}c + \alpha \Lambda u = 0,   \\
\partial_{t}u - \kappa \Delta u - \beta \Lambda c = 0,
\end{cases}
\end{split}
\end{align}
where $c$, $u$ are scalar functions and $\alpha$, $\beta$, $\kappa$ are positive constants.
We first give some important properties of the Green's matrix for the above system.
\begin{lemma}\label{green}
Let $\mathcal{G}$ be the Green matrix of system (\ref{summary}). Then we have the following explicit expression of $\hat{\mathcal{G}}$:
\begin{align*}
\hat{\mathcal{G}}(\xi,t) = \left(
                   \begin{array}{cc}
                     \frac{\lambda_{+}e^{\lambda_{-}t}-\lambda_{-}e^{\lambda_{+}t}}{\lambda_{+}-\lambda_{-}} & - \alpha \left( \frac{e^{\lambda_{+}t}-e^{\lambda_{-}t}}{\lambda_{+}-\lambda_{-}} \right)|\xi| \\
                     -\beta \left(\frac{e^{\lambda_{+}t}-e^{\lambda_{-}t}}{\lambda_{+}-\lambda_{-}}\right)|\xi| & \frac{\lambda_{+}e^{\lambda_{+}t}-\lambda_{-}e^{\lambda_{-}t}}{\lambda_{+}-\lambda_{-}} \\
                   \end{array}
                 \right)
\end{align*}
where
\begin{align*}
\lambda_{\pm} = -\frac{1}{2}\kappa |\xi|^{2} \pm \frac{1}{2} \sqrt{\kappa^{2}|\xi|^{4} - 4\alpha\beta |\xi|^{2}}.
\end{align*}
\end{lemma}

\begin{lemma}\label{estimateG}
Given $R > 0$, there is a positive number $\mathcal{\theta}$ depending on $R$ such that, for any multi-indices $\gamma$ and $|\xi| \leq R$,
\begin{align*}
|D^{\gamma}_{\xi}\hat{\mathcal{G}}(\xi,t)| \leq C e^{-\mathcal{\theta}|\xi|^{2}t}(1+|\xi|)^{|\gamma|}(1+t)^{|\gamma|}
\end{align*}
where $C = C(R, |\gamma|)$.
\end{lemma}

The proof of the above two Lemmas follows from Lemma 3.1 and Theorem 3.2 in \cite{Hoff1995}, so we omit the proof for simplicity.
Next, we prove an important Lemma which plays the key role in the low frequency analysis.
\begin{lemma}\label{KeyLemmaInLowFrequency}
Let $\mathcal{G}$ be the Green matrix of system (\ref{summary}), dimension $n = 3$. Denote $U_{0} = (c_{0}, u_{0})$, then $\mathcal{G}(t)$ satisfies the estimate
\begin{align*}
\sum_{q \leq R} \|\mathcal{G}(t)\dot{\Delta}_{q} U_{0}\|_{L^{2}} \leq C (1+t)^{-n/4} \|U_{0}\|_{\dot{B}_{1,\infty}^{0}}
\end{align*}
for $t \geq 0$ and $R > 0$ is a large positive constant.
\end{lemma}
\begin{proof}
By Placherel's theorem and Lemma \ref{estimateG}, we have
\begin{align}\label{SingleKey}
\begin{split}
\|\mathcal{G}(t)\dot{\Delta}_{q}U_{0}\|_{L^{2}} \lesssim &
\left( \int_{\frac{3}{4} 2^{q} < |\xi| < \frac{8}{3} 2^{q}} \left| e^{\hat{\mathcal{G}}(\xi)t} \phi_{q}(\xi) \hat{U}_{0} \right|^{2} d\xi \right)^{1/2} \\
\lesssim & \left( \int_{\frac{3}{4} 2^{q} < |\xi| < \frac{8}{3} 2^{q}} e^{-\theta |\xi|^{2}t} \left| \phi_{q}(\xi) \hat{U}_{0} \right|^{2}  \right)^{1/2}   \\
\lesssim & \|\dot{\Delta}_{q}U_{0}\|_{L^{1}} \left( \int_{\frac{3}{4} 2^{q} < |\xi| < \frac{8}{3} 2^{q}} e^{-\theta |\xi|^{2}t} d\xi \right)^{1/2}  \\
\lesssim & t^{-n/4} \|\dot{\Delta}_{q}U_{0}\|_{L^{1}}
\left( \int_{\sqrt{t}\frac{3}{4} 2^{q} < |\xi| < \sqrt{t}\frac{8}{3} 2^{q}} r^{\frac{n-1}{2}}e^{-\theta r^{2}} dr \right)^{1/2} \\
\lesssim & t^{-n/4} \|U_{0}\|_{\dot{B}_{1,\infty}^{0}} e^{-\frac{2}{9}4^{q}t\theta}\left( 1- e^{-\frac{20}{3}4^{q}t\theta} \right)^{1/2}
\end{split}
\end{align}
Now we do some calculations to bound $\sum_{q \leq R} e^{-\frac{2}{9}4^{q}t\theta}\left( 1- e^{-\frac{20}{3}4^{q}t\theta} \right)^{1/2}$.
Let $k = -q$, then we have
\begin{align*}
\sum_{q \leq R} e^{-\frac{2}{9}4^{q}t\theta}\left( 1- e^{-\frac{20}{3}4^{q}t\theta} \right)^{1/2} = &
\sum_{k = -R}^{\infty} e^{-\frac{2}{9}\left(\frac{1}{4}\right)^{k}t\theta}\left( 1- e^{-\frac{20}{3}\left(\frac{1}{4}\right)^{k}t\theta} \right)^{1/2}  \\
= & I + II + III,
\end{align*}
where
\begin{align*}
I = \sum_{k = -R}^{0} e^{-\frac{2}{9}\left(\frac{1}{4}\right)^{k}t\theta}\left( 1- e^{-\frac{20}{3}\left(\frac{1}{4}\right)^{k}t\theta} \right)^{1/2},
\end{align*}
\begin{align*}
II = \sum_{k = 1}^{\left[ \log_{4}\frac{20}{3}t\theta \right]}
e^{-\frac{2}{9}\left(\frac{1}{4}\right)^{k}t\theta}\left( 1- e^{-\frac{20}{3}\left(\frac{1}{4}\right)^{k}t\theta} \right)^{1/2}
\end{align*}
and
\begin{align*}
III = \sum_{k = \left[ \log_{4}\frac{20}{3}t\theta \right] + 1}^{\infty}
e^{-\frac{2}{9}\left(\frac{1}{4}\right)^{k}t\theta}\left( 1- e^{-\frac{20}{3}\left(\frac{1}{4}\right)^{k}t\theta} \right)^{1/2}.
\end{align*}
For $I$, we have $I \leq R\cdot 1 \leq C < \infty$.
Due to for arbitrary $t > 0$, there exist positive integer $N > 0$ such that $t\theta \leq \frac{3}{20} 4^{N}$.
Without loss of generality, we can choose $t\theta = \frac{3}{20} 4^{N}$.
For $II$, we have
\begin{align*}
II = & \sum_{k = 1}^{N} e^{-\frac{2}{9}\left( \frac{1}{4} \right)^{k} \frac{3}{20} 4^{N}} \left( 1 - e^{-\left( \frac{1}{4} \right)^{k}4^{N}} \right)^{1/2} \\
\leq & \, C \sum_{k = 1}^{N} e^{-\frac{1}{30}\left( \frac{1}{4} \right)^{k} 4^{N}} \leq C \sum_{k = 1}^{N} e^{-\frac{1}{30}4^{-(k-N)}} \\
\leq & \, C \sum_{m = 0}^{N-1} e^{-\frac{1}{30}4^{m}} \leq C < \infty.
\end{align*}
By Taylor's formula, we have
\begin{align*}
1-e^{-\frac{20}{3}\left( \frac{1}{4} \right)^{k}t\theta} = \frac{20}{3}\left( \frac{1}{4} \right)^{k}t\theta
+ \sum_{n = 2}^{\infty} \frac{(-1)^{n+1}}{n!} \left( \frac{20}{3}\left( \frac{1}{4} \right)^{k} t\theta \right)^{n}.
\end{align*}
When $k > \log_{4}\frac{20}{3}t\theta$, we have $\frac{20}{3}\left( \frac{1}{4} \right)^{k}t\theta < 1$.
So by the properties of alternating series, we know
\begin{align}\label{AlternatingSeries}
1 - e^{-\frac{20}{3}\left( \frac{1}{4} \right)^{k}t\theta} \leq \frac{40}{3} \left( \frac{1}{4} \right)^{k}t\theta.
\end{align}
Substituting (\ref{AlternatingSeries}) into $III$, we obtain
\begin{align*}
III \leq & \, C \sum_{k = 1}^{\left[ \log_{4}\frac{20}{3}t\theta \right]} e^{-\frac{2}{9}\left( \frac{1}{4} \right)^{k}t\theta}
\left( \frac{1}{4} \right)^{k/2} \sqrt{t\theta}    \\
\leq & \, C \sqrt{t\theta} e^{-\frac{2}{9}t\theta} \sum_{k = 1}^{\left[ \log_{4}\frac{20}{3}t\theta \right]} \left( \frac{1}{2} \right)^{k}    \\
\leq & \, C < \infty.
\end{align*}
Combining the estimates for $I, II, III$, we finally get
\begin{align}\label{PriorSumKey1}
\sum_{q \leq R} e^{-\frac{2}{9}4^{q}t\theta}\left( 1- e^{-\frac{20}{3}4^{q}t\theta} \right)^{1/2} \leq \, C < \infty,
\end{align}
where $C$ does not depend on $t$.

Summing up estimate (\ref{SingleKey}) with (\ref{PriorSumKey1}), we obtain
\begin{align}\label{SumKey1}
\begin{split}
\sum_{q \leq R} \|\mathcal{G}(t)\dot{\Delta}_{q}U_{0}\|_{L^{2}} \leq & \,
C t^{-n/4} \|U_{0}\|_{\dot{B}_{1,\infty}^{0}} \sum_{q\leq R}e^{-\frac{2}{9}4^{q}t\theta}\left( 1- e^{-\frac{20}{3}4^{q}t\theta} \right)^{1/2}   \\
\leq & \, C t^{-n/4} \|U_{0}\|_{\dot{B}_{1,\infty}^{0}}.
\end{split}
\end{align}

Similarly, we also find that
\begin{align}\label{SumKey2}
\begin{split}
\sum_{q \leq R} \|\mathcal{G}(t)\dot{\Delta}_{q}U_{0}\|_{L^{2}} \leq & \,
C \sum_{q \leq R} \|\dot{\Delta}_{q}U_{0}\|_{L^{1}} \left( \int_{\frac{3}{4}2^{q} \leq r \leq \frac{8}{3}2^{q}} r^{n-1} e^{-\theta r^{2}t} dr \right)^{1/2} \\
\leq & \, C \|U_{0}\|_{\dot{B}_{1,\infty}^{0}} \sum_{q \leq R}(\sqrt{8})^{q} \leq C < \infty.
\end{split}
\end{align}
Combining (\ref{SumKey1}) and (\ref{SumKey2}), we finally arrive our desired results.
\end{proof}

\begin{remark}\label{G1G2G3}
Denote $\mathcal{G}_{1}$, $\mathcal{G}_{2}$ and $\mathcal{G}_{3}$ represent the Green matrix of system (\ref{linear1}),
(\ref{linear2}) and (\ref{linear3}) respectively.
Denote $V_{0}^{1} = (a_{0}, d_{0})$, $V_{0}^{2} = (\mathcal{E}_{0}, d_{0})$ and $V_{0}^{3} = (F_{0}^{T} - F_{0}, \Omega_{0})$, then
using similar methods as in the proof of Lemma \ref{KeyLemmaInLowFrequency}, we will have
\begin{align*}
\sum_{q \leq R} \|\mathcal{G}_{i}(t)\dot{\Delta}_{q} V_{0}^{i}\|_{L^{2}} \leq C (1+t)^{-n/4} \|V_{0}^{i}\|_{\dot{B}_{1,\infty}^{0}}
\end{align*}
for $i = 1, 2, 3$.
\end{remark}

Next, we need to consider estimates about $M_{1}$, $M_{2}$, $M_{3}$ defined in (\ref{DefineF1F2}) and (\ref{DefineF3}).
\begin{lemma}\label{RightHandSide}
There exist an $\delta > 0$ such that if
\begin{align*}
\|a_{0}\|_{B^{n/2}} + \|v_{0}\|_{B^{n/2-1}} + \|F_{0}\|_{B^{n/2}} \leq \delta,
\end{align*}
then there exists a constant $C > 0$ independent of time $T$ such that
\begin{align*}
\|M_{1}, M_{2}, M_{3}\|_{\dot{B}_{1,\infty}^{0}} \leq \,
C (1+t)^{-n/4}M(t)f(t) + C (1+t)^{-n/2}M^2(t)
\end{align*}
for $t \in [0, T]$, where $f(t) = \|v(t)\|_{\dot{B}^{n/2+1}} \in L^{1}(0, \infty)$.
\end{lemma}
\begin{proof}
Now, we start with $M_{1}$.
For term $v\cdot\nabla a$, we have
\begin{align}\label{F1_1}
\begin{split}
\|v\cdot\nabla a\|_{\dot{B}_{1,\infty}^{0}} \leq & \, C\, \|v\cdot\nabla a\|_{L^{1}} \leq \, C\, \|v\|_{L^{2}} \|\nabla a\|_{L^{2}}   \\
\leq & \, C\, (1+t)^{-n/4} M_{4}(t)\|a\|_{B^{n/2-1,n/2}}    \\
\leq & \, C\, (1+t)^{-n/2}M^{2}(t).
\end{split}
\end{align}
For term $a\mathrm{div}v$, we have
\begin{align}\label{F1_2}
\begin{split}
\|a\mathrm{div}v\|_{\dot{B}_{1,\infty}^{0}} \leq & \, C\, \|a\mathrm{div}v\|_{L^{1}} \leq \, C\, \|a\|_{L^{2}}\|\nabla v\|_{L^{2}}  \\
\leq & \, C\, (1+t)^{-n/4}M_{4}(t) \left( \sum_{q\leq R}\|\dot{\Delta}_{q}\nabla v\|_{L^{2}} + \sum_{q>R}\|\dot{\Delta}_{q}\nabla v\|_{L^{2}} \right)   \\
\leq & \, C\, (1+t)^{-n/4}M_{4}(t) \left( \|v\|_{L^{2}} + \|v\|_{\dot{B}^{n/2+1}} \right)   \\
\leq & \, C\, (1+t)^{-n/2}M_{4}(t) \left( M_{4}(t) + f(t) \right)   \\
\leq & \, C\, (1+t)^{-n/2}M^{2}(t) + C\, (1+t)^{-n/4}M(t)f(t).
\end{split}
\end{align}
For term $v\cdot\nabla d$, we have
\begin{align}\label{F1_3}
\begin{split}
\|v\cdot\nabla d\|_{\dot{B}_{1,\infty}^{0}} \leq & \, C\, \|v\cdot\nabla d\|_{L^{1}} \leq \, C\, \|v\|_{L^{2}} \|\nabla d\|_{L^{2}}   \\
\leq & \, C\, \|v\|_{L^{2}}\left( \|v\|_{L^{2}} + \|v\|_{\dot{B}^{n/2+1}} \right)   \\
\leq & \, C\, (1+t)^{-n/2}M^{2}(t) + C\, (1+t)^{-n/4}M(t)f(t).
\end{split}
\end{align}
For term $\Lambda^{-1}\mathrm{div}(v\cdot\nabla v)$, we have
\begin{align}\label{F1_4}
\begin{split}
\|\Lambda^{-1}\mathrm{div}(v\cdot\nabla v)\|_{\dot{B}_{1,\infty}^{0}} \leq & \, C\, \|v\cdot\nabla v\|_{\dot{B}_{1,\infty}^{0}}
\leq \, C\, \|v\cdot\nabla v\|_{L^{1}}    \\
\leq & \, C\, (1+t)^{-n/2}M^{2}(t) + C\, (1+t)^{-n/4}M(t)f(t),
\end{split}
\end{align}
where we used similar argument in (\ref{F1_3}) to get the last inequality.
For term $\Lambda^{-1}\mathrm{div}(F\nabla F)$, using similar method as in (\ref{F1_4}) and (\ref{F1_1}), we have
\begin{align}\label{F1_5}
\begin{split}
\|\Lambda^{-1}\mathrm{div}(F\nabla F)\|_{\dot{B}_{1,\infty}^{0}} \leq \, C\, (1+t)^{-n/2}M^{2}(t).
\end{split}
\end{align}
Using composition rules (for example: Theorem 2.61 in \cite{FourierPDE}) and similar argument as above, we obtain
\begin{align}\label{F1_6}
\begin{split}
\left\| \Lambda^{-1}\mathrm{div}\left( \frac{a}{1+a}\mathcal{A}v \right) \right\|_{\dot{B}_{1,\infty}^{0}}
\leq & \, C\, (1+t)^{-n/2}M^{2}(t)    \\
& \quad\quad + C\, (1+t)^{-n/4}M(t)f(t).
\end{split}
\end{align}
Summing up estimates (\ref{F1_1}) to (\ref{F1_6}), we get
\begin{align}\label{F1Estimate}
\begin{split}
\|M_{1}\|_{\dot{B}_{1,\infty}^{0}} \leq \, C\, (1+t)^{-n/2}M^{2}(t) + C\, (1+t)^{-n/4}M(t)f(t).
\end{split}
\end{align}
Next, let us due with $F_2$.
The term $v\cdot\nabla \mathcal{E}$, $v\cdot\nabla d$ and $K$ all can be estimated similar to
the term appeared in $F_1$, so we just need to give the following estimates about $J$.
Since
\begin{align}\label{F2_J1}
\begin{split}
& \left\|\Lambda^{-1}\partial_{x_{i}}\Lambda^{-1}\nabla_{x_{j}}\left[ \left( \nabla v F \right)^{ij} + \left( \nabla v F \right)^{ji} \right] \right\|_{\dot{B}_{1,\infty}^{0}} \leq \, C\, \|\nabla v\|_{L^{2}} \|F\|_{L^{2}}    \\
& \quad \leq \, C\, (1+t)^{-n/2}M^{2}(t) + C\, (1+t)^{-n/4}M(t)f(t),
\end{split}
\end{align}
and
\begin{align}\label{F2_J2}
\begin{split}
& \| \left[ \Lambda^{-1}\partial_{x_{i}}\Lambda^{-1}\partial_{x_{j}}, v^{k} \right] \partial_{x_{k}} (F^{ij} + F^{ji}) \|_{\dot{B}_{1,\infty}^{0}}    \\
\leq & \|\Lambda^{-1}\partial_{x_{i}}\Lambda^{-1}\partial_{x_{j}}(v^{k}\partial_{x_{k}}(F^{ij} + F^{ji}))\|_{\dot{B}_{1,\infty}^{0}}    \\
& \quad\quad\quad\quad\quad\quad\quad
+ \|v^{k}\cdot\Lambda^{-1}\partial_{x_{i}}\Lambda^{-1}\partial_{x_{j}}(\partial_{x_{k}}(F^{ij} + F^{ji}))\|_{\dot{B}_{1,\infty}^{0}}    \\
\leq & \, C\, (1+t)^{-n/2}M^{2}(t) + C\, \|v\|_{L^{2}}
\|\Lambda^{-1}\partial_{x_{i}}\Lambda^{-1}\partial_{x_{j}}(\partial_{x_{k}}(F^{ij} + F^{ji}))\|_{\dot{B}_{1,\infty}^{0}}    \\
\leq & \, C\, (1+t)^{-n/2}M^{2}(t) + C\, \|v\|_{L^{2}} \left( \|F\|_{L^{2}} + \|F\|_{B^{n/2-1,n/2}} \right) \\
\leq & C\, (1+t)^{-n/2}M^{2}(t),
\end{split}
\end{align}
we have
\begin{align}\label{F2Estimate}
\begin{split}
\|M_{2}\|_{\dot{B}_{1,\infty}^{0}} \leq \, C\, (1+t)^{-n/2}M^{2}(t) + C\, (1+t)^{-n/4}M(t)f(t).
\end{split}
\end{align}
Due to all terms appeared in $M_{3}$ can be estimated similar to the terms appeared in $M_{1}$ and $M_{2}$, here, we just give the
estimates as follows
\begin{align}\label{F3Estimate}
\begin{split}
\|M_{3}\|_{\dot{B}_{1,\infty}^{0}} \leq \, C\, (1+t)^{-n/2}M^{2}(t) + C\, (1+t)^{-n/4}M(t)f(t).
\end{split}
\end{align}
At this stage, we easily finished the proof by just summing up (\ref{F1Estimate}), (\ref{F2Estimate})
and (\ref{F3Estimate}).
\end{proof}

Denote $V = (a, v, F)$, $V_{1} = (a, d)$, $V_{2} = (\mathcal{E}, d)$, $V_{3} = (F^{T}-F, \Omega)$
and define $\Delta_{R}f := \Delta_{-1}f + \sum_{0\leq q \leq R}\Delta_{q}f$ for a tempered distribution $f$.
Now, we can prove the following proposition which is the main results for low frequency part.
\begin{proposition}\label{LowFrequencyPro}
Let $n = 3$, there exists an $\delta > 0$ such that if
\begin{align*}
\|a_{0}\|_{B^{n/2}} + \|v_{0}\|_{B^{n/2-1}} + \|F_{0}\|_{B^{n/2}} \leq \delta,
\end{align*}
then there exists a constant $C > 0$ independent of time $T$ such that
\begin{align*}
\sup_{0 \leq \tau \leq t}(1+\tau)^{n/4}\|\Delta_{R} V(\tau)\|_{L^{2}} \leq \, C\, \|V_{0}\|_{\dot{B}_{1,\infty}^{0}} + C\, \delta M(t)
+ C\, M^{2}(t)
\end{align*}
for $t \in [0,T]$.
\end{proposition}
\begin{proof}
By the properties of Littlewood-Paley operator and (5.13), (5.14), (5.15) in \cite{Jia2014}, we have
\begin{align}\label{LowFrePro1}
\begin{split}
\|\Delta_{R}V(\tau)\|_{L^{2}} \leq & \sum_{q \leq R} \|\dot{\Delta}_{q}\Delta_{R}V(\tau)\|_{L^{2}} \lesssim \sum_{q \leq R} \|\dot{\Delta}_{q}V(\tau)\|_{L^{2}}   \\
\lesssim & \, \sum_{q \leq R}\|\dot{\Delta}_{q}V_{1}(\tau)\|_{L^{2}}
+ \sum_{q \leq R}\|\dot{\Delta}_{q}V_{2}(\tau)\|_{L^{2}}    \\
& \, + \sum_{q \leq R}\|\dot{\Delta}_{q}V_{3}(\tau)\|_{L^{2}} + \sum_{q \leq R}\|\dot{\Delta}_{q}\Lambda^{-1}(F\nabla F)(\tau)\|_{L^{2}}
\end{split}
\end{align}
For the last term appeared in the above inequality (\ref{LowFrePro1}), we have
\begin{align}\label{LowFrePro2}
\begin{split}
\sum_{q \leq R}\|\dot{\Delta}_{q}\Lambda^{-1}(F\nabla F)(\tau)\|_{L^{2}} \leq & \, C\, \|F\|_{\dot{B}_{2,2}^{0}} \|F\|_{\dot{B}^{n/2}}   \\
\leq & \, C\, (1+t)^{-n/2}M^{2}(t),
\end{split}
\end{align}
where we used Lemma A.4 in \cite{Jia2014}(take $\tilde{t} = s = 0$, $\tilde{s} = t = \frac{1}{2}$, $p = 2$ and $\gamma = 0$).
From (\ref{newform}) and (\ref{equi}), we easily get
\begin{align}\label{IntegralForm}
\begin{split}
& \dot{\Delta}_{q}V_{1}(t) = \mathcal{G}_{1}(t)\dot{\Delta}_{q}V_{10} + \int_{0}^{t} \mathcal{G}_{1}(t-s)\dot{\Delta}_{q}M_{1}(s) ds,    \\
& \dot{\Delta}_{q}V_{2}(t) = \mathcal{G}_{2}(t)\dot{\Delta}_{q}V_{20} + \int_{0}^{t} \mathcal{G}_{2}(t-s)\dot{\Delta}_{q}M_{2}(s) ds,   \\
& \dot{\Delta}_{q}V_{3}(t) = \mathcal{G}_{3}(t)\dot{\Delta}_{q}V_{30} + \int_{0}^{t} \mathcal{G}_{3}(t-s)\dot{\Delta}_{q}M_{3}(s) ds.
\end{split}
\end{align}
By using Remark \ref{G1G2G3} and Lemma \ref{RightHandSide}, we will get
\begin{align}\label{LowFrePro3}
\begin{split}
\|\Delta_{R}V(\tau)\|_{L^{2}} \lesssim & \, \sum_{q \leq R}\|\dot{\Delta}_{q}V_{1}(\tau)\|_{L^{2}} +
\sum_{q \leq R}\|\dot{\Delta}_{q}V_{2}(\tau)\|_{L^{2}}  \\
& \, + \sum_{q \leq R}\|\dot{\Delta}_{q}V_{3}(\tau)\|_{L^{2}} + (1+\tau)^{-n/2}M^{2}(\tau)  \\
\lesssim & \, I + II + (1+\tau)^{-n/2}M^{2}(\tau),
\end{split}
\end{align}
where
\begin{align}\label{LowFrePro4}
\begin{split}
I = & \sum_{q \leq R}\left\{ \|\mathcal{G}_{1}(\tau)\dot{\Delta}_{q}V_{1}(\tau)\|_{L^{2}}
+ \|\mathcal{G}_{2}(\tau)\dot{\Delta}_{q}V_{2}(\tau)\|_{L^{2}} + \|\mathcal{G}_{3}(\tau)\dot{\Delta}_{q}V_{3}(\tau)\|_{L^{2}} \right\}  \\
\leq & \, C\, (1+\tau)^{-n/4} \left\{ \|V_{10}\|_{\dot{B}_{1,\infty}^{0}} + \|V_{20}\|_{\dot{B}_{1,\infty}^{0}}
+ \|V_{30}\|_{\dot{B}_{1,\infty}^{0}} \right\}  \\
\leq & \, C\, (1+\tau)^{-n/4} \|V_{0}\|_{\dot{B}_{1,\infty}^{0}},
\end{split}
\end{align}
and
\begin{align}\label{LowFrePro5}
\begin{split}
II = & \sum_{q \leq R}\int_{0}^{\tau}\Big\{ \|\mathcal{G}_{1}(\tau - s)\dot{\Delta}_{q}V_{1}(s)\|_{L^{2}}
+ \|\mathcal{G}_{2}(\tau - s)\dot{\Delta}_{q}V_{2}(s)\|_{L^{2}}     \\
& \quad\quad\quad\quad\,
+ \|\mathcal{G}_{3}(\tau - s)\dot{\Delta}_{q}V_{3}(s)\|_{L^{2}} \Big\}ds    \\
\leq & \, C\, \int_{0}^{\tau} (1+\tau - s)^{-n/4} \left\{ \|V_{1}\|_{\dot{B}_{1,\infty}^{0}} + \|V_{2}\|_{\dot{B}_{1,\infty}^{0}}
+ \|V_{3}\|_{\dot{B}_{1,\infty}^{0}} \right\} ds    \\
\leq & \, C\, (1+\tau)^{-n/4}M(t)\int_{0}^{\tau}f(s)ds + C\, (1+\tau)^{-n/2}M^{2}(\tau) \\
\leq & \, C\, (1+\tau)^{-n/4}M(t)\delta + C\, (1+\tau)^{-n/2}M^{2}(\tau).
\end{split}
\end{align}
Combining (\ref{LowFrePro3}), (\ref{LowFrePro4}) and (\ref{LowFrePro5}), we finally obtain
\begin{align}\label{LowFrePro6}
\begin{split}
\sup_{0 \leq \tau \leq t}(1+\tau)^{n/4}\|\Delta_{R}V(\tau)\|_{L^{2}} \leq \, C\, \|V_{0}\|_{\dot{B}_{1,\infty}^{0}}
+ C\, \delta M(t) + C \, M^{2}(t).
\end{split}
\end{align}
\end{proof}

\section{Analysis about high frequency part}

In this part, we need to transform the equation into another form and estimate in the high frequency domain
which is completely different to the method used for the low frequency domain.

Without loss of generality, assume $\bar{\rho} = 1$ and $\gamma = \sqrt{P'(\bar{\rho})} - 1$.
Denote $a = \rho - 1$, $F = U - I$, $\Lambda = (-\Delta)^{1/2}$, $d = \Lambda^{-1}\mathrm{div}v$, $e^{ij} = \Lambda^{-1}\nabla_{j}v^{i}$.
From $U^{\ell k}\nabla_{\ell}U^{ij} - U^{\ell j}\nabla_{\ell}U^{ik} = 0$, we easily know
\begin{align*}
\Lambda^{-1}(\nabla_{j}\nabla_{k}F^{ik}) =
-\Lambda F^{ij} - \Lambda^{-1}\nabla_{k} \left( F^{\ell j}\nabla_{\ell}F^{ik} - F^{\ell k}\nabla_{\ell}F^{ij} \right).
\end{align*}
Hence, we can transform the equations (\ref{Viscoelastic 2}) into the following new form.
\begin{align}\label{HighFreForm1}
\begin{split}
& \partial_{t}a + v\cdot\nabla a + \Lambda d = G_{1},   \\
& \partial_{t}e^{ij} + v\cdot\nabla e^{ij} - \mu \Delta e^{ij} - (\lambda + \mu)\nabla_{i}\nabla_{j}d    \\
& \quad\quad\quad\quad\quad\quad\quad\quad\quad\quad\quad\quad
+ \Lambda^{-1}\nabla_{i}\nabla_{j}a + \Lambda F^{ij} = G_{2}^{ij}   \\
& \partial_{t}F^{ij} + v\cdot\nabla F^{ij} - \Lambda e^{ij} = G_{3}^{ij},
\end{split}
\end{align}
where
\begin{align*}
G_{1} = a\,\mathrm{div}v, \quad G_{3}^{ij} = \nabla_{k}v^{i}\, F^{kj},
\end{align*}
and
\begin{align*}
G_{2}^{ij} = & v\cdot\nabla e^{ij} - \Lambda^{-1}\nabla_{j}\Big[ v\cdot\nabla v^{i} + C(a)\mathcal{A}v + F^{jk}\nabla_{j}F^{ik} \Big] \\
& + \Lambda^{-1}\nabla_{k}(F^{\ell j}\nabla F^{ik} - F^{\ell k}\nabla_{\ell}F^{ij})
\end{align*}
with $C(a) = \frac{a}{1+a}$, $K(a) = \frac{P'(1+a)}{1+a} - 1$.
Moreover, we have
\begin{align}\label{HighFreForm2}
\begin{split}
\nabla_{i}F^{ij} = -\nabla_{j}a + G_{0}^{j}, \quad G_{0}^{j} = - \nabla_{i}(a F^{ij}).
\end{split}
\end{align}

Now, we give the main estimates for high frequency domain in the following proposition.
\begin{proposition}\label{HighMainPro}
There exists an $\delta > 0$ such that if
\begin{align*}
\|a_{0}\|_{B^{n/2}} + \|v_{0}\|_{B^{n/2-1}} + \|F_{0}\|_{B^{n/2}} \leq \delta,
\end{align*}
then there holds
\begin{align*}
\frac{d}{dt}E_{q}(t) + c_{0}E_{q}(t) \leq & \, C\, \Big\{ \alpha_{q}(1+t)^{-n/4}M(t)f(t)      \\
& \quad\quad
+ \alpha_{q}\|G_{1}, G_{3}\|_{B^{n/2-1,n/2}} + \alpha_{q} \|G_{0}, G_{2}\|_{\dot{B}^{n/2-1}} \Big\}
\end{align*}
for $t \in [0,T]$ and $q \geq R$, where $\sum_{q\geq 1}\alpha_{q} \leq 1$,
\begin{align*}
\int_{0}^{\infty}f(t)dt = \int_{0}^{\infty}\|v(t)\|_{\dot{B}^{n/2+1}}dt \leq C\delta
\end{align*}
and $c_{0}$ dose not depend on $q$.
Here, $E_{q}(t)$ is equivalent to
$2^{\frac{n}{2}q} \|\dot{\Delta}_{q}a\|_{L^{2}} + 2^{\frac{n}{2}q} \|\dot{\Delta}_{q}F\|_{L^{2}} + 2^{\left(\frac{n}{2}-1\right)q} \|\dot{\Delta}_{q}e\|_{L^{2}}$.
That is, there exists a $D_{1}$ such that
\begin{align*}
\frac{1}{D_{1}}\tilde{E}_{q}
\leq E_{q} \leq
D_{1} \tilde{E}_{q}
\end{align*}
where
\begin{align*}
\tilde{E}_{q} = 2^{\frac{n}{2}q} \|\dot{\Delta}_{q}a, \dot{\Delta}_{q}F\|_{L^{2}} + 2^{\left(\frac{n}{2}-1\right)q} \|\dot{\Delta}_{q}e\|_{L^{2}}
\end{align*}
\end{proposition}
\begin{proof}
Applying the operator $\dot{\Delta}_{q}$ to system (\ref{HighFreForm1}), we find that $(a, e, F)$ satisfies
\begin{align}\label{LocalHigFreFrom1}
\begin{split}
& \dot{\Delta}_{q}\partial_{t}a + \Lambda\dot{\Delta}_{q}d = \dot{\Delta}_{q}G_{1} - \dot{\Delta}_{q}(v\cdot\nabla a),    \\
& \dot{\Delta}_{q}\partial_{t}e^{ij} - \mu \Delta\dot{\Delta}_{q}e^{ij} - (\lambda + \mu)\nabla_{i}\nabla_{j}\dot{\Delta}_{q}d
+ \Lambda^{-1}\nabla_{i}\nabla_{j}\dot{\Delta}_{q}a \\
& \quad\quad\quad\quad\quad\quad\quad\quad\quad\quad\quad\quad\quad
+ \Lambda\dot{\Delta}_{q}F^{ij} = \dot{\Delta}_{q}G_{2}^{ij} - \dot{\Delta}_{q}(v\cdot\nabla e^{ij}),       \\
& \dot{\Delta}_{q}\partial_{t}F^{ij} - \Lambda\dot{\Delta}_{q}e^{ij} = \dot{\Delta}_{q}G_{3}^{ij} - \dot{\Delta}_{q}(v\cdot\nabla F^{ij}),
\end{split}
\end{align}
where $i,j = 1, 2, 3$.
Taking the $L^{2}$-product of the second equation of (\ref{LocalHigFreFrom1}) with $\dot{\Delta}_{q}e^{ij}$, then
summing up the resulting equation with respect to indexes $i,j$, we can get
\begin{align}\label{H6.5}
\begin{split}
& \frac{1}{2}\frac{d}{dt}\|\dot{\Delta}_{q}e\|_{L^{2}}^{2} + \mu\|\Lambda\dot{\Delta}_{q}e\|_{L^{2}}^{2} + (\lambda + \mu)\|\Lambda\dot{\Delta}_{q}d\|_{L^{2}}^{2}
- (\dot{\Delta}_{q}a | \Lambda\dot{\Delta}_{q}d)    \\
& \quad\quad\quad\quad\quad\quad\quad\quad
+ (\Lambda\dot{\Delta}_{q}F | \dot{\Delta}_{q}a) = (\dot{\Delta}_{q}G_{2} | \dot{\Delta}_{q}e) - (\dot{\Delta}_{q}(v\cdot\nabla e) | \dot{\Delta}_{q}e),
\end{split}
\end{align}
where we used the fact $d = - \Lambda^{-2}\nabla_{i}\nabla_{j}e^{ij}$.
We apply the operator $\Lambda$ to the first equation of (\ref{LocalHigFreFrom1}) and take the $L^{2}$-product of the
resulting equation with $-\dot{\Delta}_{q}d$, and take the $L^{2}$-product of the second equation of (\ref{LocalHigFreFrom1})
with $\Lambda^{-1}\nabla_{i}\nabla_{j}\dot{\Delta}_{q}a$. Then, summing up the resulting equations yields that
\begin{align}\label{H6.7}
\begin{split}
& -\frac{d}{dt}(\Lambda\dot{\Delta}_{q}a | \dot{\Delta}_{q}d) - \|\Lambda\dot{\Delta}_{q}d\|_{L^{2}}^{2} + \|\Lambda\dot{\Delta}_{q}a\|_{L^{2}}^{2}
- (\lambda + 2\mu) (\Lambda^{2}\dot{\Delta}_{q}d | \Lambda\Delta_{q}a)  \\
& \quad\quad
+ (\dot{\Delta}_{q}F^{ij} | \nabla_{i}\nabla_{j}\dot{\Delta}_{q}a) = - (\Lambda\dot{\Delta}_{q}G_{1} | \dot{\Delta}_{q}d)
+ (\dot{\Delta}_{q}G_{2}^{ij} | \Lambda^{-1}\nabla_{i}\nabla_{j}\dot{\Delta}_{q}a)  \\
& \quad\quad
+ (\Lambda\dot{\Delta}_{q}(v\cdot\nabla a) | \dot{\Delta}_{q}d) - (\dot{\Delta}_{q}(v\cdot\nabla e^{ij}) | \Lambda^{-1}\nabla_{i}\nabla_{j}\dot{\Delta}_{q}a).
\end{split}
\end{align}
We apply the operator $\Lambda$ to the third equation of (\ref{LocalHigFreFrom1}) and take the $L^{2}$-product of
the resulting equation with $\dot{\Delta}_{q}e^{ij}$ and take the $L^{2}$-product of the second equation of (\ref{LocalHigFreFrom1})
with $\Lambda\dot{\Delta}_{q}F^{ij}$. Then, summing up the resulting equations yields that
\begin{align}\label{H6.8}
\begin{split}
& \frac{d}{dt}(\Lambda\dot{\Delta}_{q}F | \dot{\Delta}_{q}e) - \|\Lambda\dot{\Delta}_{q}e\|_{L^{2}}^{2} + \|\Lambda\dot{\Delta}_{q}F\|_{L^{2}}^{2}
+ \mu (\Lambda^{2}\dot{\Delta}_{q}e | \Lambda\dot{\Delta}_{q}F) \\
& \quad\quad\quad
+ (\lambda + \mu)(\nabla_{i}\nabla_{j}\dot{\Delta}_{q}d | \Lambda\dot{\Delta}_{q}F^{ij}) + (\nabla_{i}\nabla_{j}\dot{\Delta}_{q}a | \dot{\Delta}_{q}F^{ij}) \\
& \quad\quad\quad
= (\dot{\Delta}_{q}G_{2} | \Lambda\dot{\Delta}_{q}F) + (\Lambda\dot{\Delta}_{q}G_{3} | \dot{\Delta}_{q}e)
- (\Lambda\dot{\Delta}_{q}(v\cdot\nabla e) | \dot{\Delta}_{q}F)     \\
& \quad\quad\quad
- (\Lambda\dot{\Delta}_{q}(v\cdot\nabla F) | \dot{\Delta}_{q}e).
\end{split}
\end{align}
Now, applying the operator $\Lambda$ to the first and the third equations of (\ref{LocalHigFreFrom1}),
then taking the $L^{2}$ product of the resulting equations with $\Lambda\dot{\Delta}_{q}a$ and $\Lambda\dot{\Delta}_{q}F^{ij}$, we will get
\begin{align}\label{H6.10}
\begin{split}
& \frac{1}{2}\frac{d}{dt}\|\Lambda\dot{\Delta}_{q}a\|_{L^{2}}^{2} + (\Lambda^{2}\dot{\Delta}_{q}d | \Lambda\dot{\Delta}_{q}a) \\
& \quad\quad\quad\quad\quad\quad\quad\quad
= (\Lambda\dot{\Delta}_{q}G_{1} | \Lambda\dot{\Delta}_{q}a) - (\Lambda\dot{\Delta}_{q}(v\cdot\nabla a) | \Lambda\dot{\Delta}_{q}a),
\end{split}
\end{align}
and
\begin{align}\label{H6.11}
\begin{split}
& \frac{1}{2}\frac{d}{dt}\|\Lambda\dot{\Delta}_{q}F\|_{L^{2}}^{2} - (\Lambda^{2}\dot{\Delta}_{q}e | \Lambda\dot{\Delta}_{q}F) \\
& \quad\quad\quad\quad\quad\quad\quad\quad
= (\Lambda\dot{\Delta}_{q}G_{3} | \Lambda\dot{\Delta}_{q}F) - (\dot{\Delta}_{q}(v\cdot\nabla F) | \Lambda\dot{\Delta}_{q}F).
\end{split}
\end{align}
We apply the operator $\Lambda^{-1}\nabla_{i}\nabla_{j}$ to the third equation of (\ref{LocalHigFreFrom1}) and take
the summation with respect to $i,j$, then we take the $L^{2}$ times the resulting equation with $\Lambda^{-1}\nabla_{i}\nabla_{j}\dot{\Delta}_{q}F^{ij}$
to get
\begin{align}\label{H6.12}
\begin{split}
& \frac{1}{2}\frac{d}{dt}\|\Lambda^{-1}\nabla_{i}\nabla_{j}\dot{\Delta}_{q}F^{ij}\|_{L^{2}}^{2} +
(\Lambda\dot{\Delta}_{q}d | \nabla_{i}\nabla_{j}\dot{\Delta}_{q}F^{ij})     \\
& \quad\quad\quad\quad
= (\Lambda^{-1}\nabla_{i}\nabla_{j}\dot{\Delta}_{q}G_{3}^{ij} | \Lambda^{-1}\nabla_{k}\nabla_{\ell}\dot{\Delta}_{q}F^{k,\ell})  \\
& \quad\quad\quad\quad\quad
- (\Lambda^{-1}\nabla_{i}\nabla_{j}\dot{\Delta}_{q}(v\cdot\nabla F^{ij}) | \Lambda^{-1}\nabla_{k}\nabla_{\ell}\dot{\Delta}_{q}F^{k,\ell}).
\end{split}
\end{align}
Summing up (\ref{H6.5}), (\ref{H6.7}), (\ref{H6.8}) and (\ref{H6.10})-(\ref{H6.12}) yields that
\begin{align}\label{H6.13}
\begin{split}
& \frac{1}{2}\frac{d}{dt}f_{q}^{2} + \tilde{f}_{q}^{2} + 2(\dot{\Delta}_{q}F^{ij} | \nabla_{i}\nabla_{j}\dot{\Delta}_{q}a) = (\dot{\Delta}_{q}G_{2} | \dot{\Delta}_{q}e)
- (\Lambda\dot{\Delta}_{q}G_{1} | \dot{\Delta}_{q}d)    \\
& \quad\quad
- (\dot{\Delta}_{q}G_{2}^{ij} | \Lambda^{-1}\nabla_{i}\nabla_{j}\dot{\Delta}_{q}a) + (\dot{\Delta}_{q}G_{2} | \Lambda\dot{\Delta}_{q}F)
+ (\Lambda\dot{\Delta}_{q}G_{3} | \dot{\Delta}_{q}e)    \\
& \quad\quad
+ (\lambda + 2\mu)(\Lambda\dot{\Delta}_{q}G_{1} | \Lambda\dot{\Delta}_{q}a)
+ \mu (\Lambda\dot{\Delta}_{q}G_{3} | \Lambda\dot{\Delta}_{q}F)    \\
& \quad\quad
+ (\lambda + \mu)(\Lambda^{-1}\nabla_{i}\nabla_{j}\dot{\Delta}_{q}G_{3} | \Lambda^{-1}\nabla_{k}\nabla_{\ell}\dot{\Delta}_{q}F^{k\ell})  + F_{q},
\end{split}
\end{align}
where
\begin{align*}
f_{q}^2 = & \|\dot{\Delta}_{q}e\|_{L^{2}}^{2} + (\lambda +2\mu)\|\Lambda\dot{\Delta}_{q}a\|_{L^{2}}^{2} + \mu\|\Lambda\dot{\Delta}_{q}F\|_{L^{2}}^{2} \\
& + (\lambda + \mu)\|\Lambda^{-1}\nabla_{i}\nabla_{j}\dot{\Delta}_{q}F^{ij}\|_{L^{2}}^{2} - 2(\Lambda\dot{\Delta}_{q}a | \dot{\Delta}_{q}d)
+ 2(\Lambda\dot{\Delta}_{q}F | \dot{\Delta}_{q}e),
\end{align*}
\begin{align*}
\tilde{f}_{q}^{2} = & (\mu - 1)\|\Lambda\dot{\Delta}_{q}e\|_{L^{2}}^{2} + (\lambda+\mu-1)\|\Lambda\dot{\Delta}_{q}d\|_{L^{2}}^{2}
+ \|\Lambda\dot{\Delta}_{q}a\|_{L^{2}}^{2} + \|\Lambda\dot{\Delta}_{q}F\|_{L^{2}}^{2}   \\
& - (\dot{\Delta}_{q}a | \Lambda\dot{\Delta}_{q}d) + (\Lambda\dot{\Delta}_{q}F | \dot{\Delta}_{q}e),
\end{align*}
\begin{align*}
F_{q} = & -(\dot{\Delta}_{q}(v\cdot\nabla e) | \dot{\Delta}_{q}e) + \Big( (\Lambda\dot{\Delta}_{q}(v\cdot\nabla a) | \dot{\Delta}_{q}d)
+ (\dot{\Delta}_{q}(v\cdot\nabla e^{ij}) | \Lambda^{-1}\nabla_{i}\nabla_{j}\dot{\Delta}_{q}a) \Big) \\
& + \Big( (\Lambda\dot{\Delta}_{q}(v\cdot\nabla e) | \dot{\Delta}_{q}F) - (\Lambda\dot{\Delta}_{q}(v\cdot\nabla F) | \dot{\Delta}_{q}e) \Big)
- \mu (\Lambda\dot{\Delta}_{q}(v\cdot\nabla F) | \Lambda\dot{\Delta}_{q}F)  \\
& - (\lambda + \mu)(\Lambda^{-1}\nabla_{i}\nabla_{j}\dot{\Delta}_{q}(v\cdot\nabla F) | \Lambda^{-1}\nabla_{k}\nabla_{\ell}\dot{\Delta}_{q}F^{k\ell})    \\
& - (\lambda + 2\mu)^{2}(\Lambda\dot{\Delta}_{q}(v\cdot\nabla a) | \nabla\dot{\Delta}_{q}a)
\end{align*}
Here, we can take $R$ to be a fix large enough constant. For $q > R$, we can easily deduce
\begin{align}\label{dengjia1}
\begin{split}
f_{q}^{2} \approx 2^{2q} \|\dot{\Delta}_{q}a\|_{L^{2}}^{2} + \|\dot{\Delta}_{q}e\|_{L^{2}}^{2} + 2^{2q}\|\dot{\Delta}_{q}F\|_{L^{2}}^{2},
\end{split}
\end{align}
and
\begin{align}\label{dengjia2}
\begin{split}
2^{2q} \|\dot{\Delta}_{q}a\|_{L^{2}}^{2} + \|\dot{\Delta}_{q}e\|_{L^{2}}^{2} + 2^{2q}\|\dot{\Delta}_{q}F\|_{L^{2}}^{2} \lesssim \,\tilde{f}_{q}^{2}.
\end{split}
\end{align}
Using the identity (\ref{HighFreForm2}), we find that
\begin{align}\label{dengjia3}
\begin{split}
(\dot{\Delta}_{q}F^{ij} | \nabla_{i}\nabla_{j}a) = & (\nabla_{i}\nabla_{j}\dot{\Delta}_{q}F^{ij} | \dot{\Delta}_{q}a)   \\
= & \|\Lambda\dot{\Delta}_{q}a\|_{L^{2}}^{2} + (\Lambda\dot{\Delta}_{q}a | \Lambda^{-1}\nabla_{j}G_{0}^{j}).
\end{split}
\end{align}
Let $E_{q}(t) = 2^{\left( \frac{n}{2} - 1 \right)q} f_{q}$, then we have
\begin{align}\label{dengjia4}
\begin{split}
E_{q}(t) \approx 2^{\frac{n}{2}q} \|\dot{\Delta}_{q}a\|_{L^{2}} + 2^{\left(\frac{n}{2}-1\right)q} \|\dot{\Delta}_{q}e\|_{L^{2}}
+ 2^{\frac{n}{2}q} \|\dot{\Delta}_{q}F\|_{L^{2}}
\end{split}
\end{align}
By (\ref{H6.13}), (\ref{dengjia1})-(\ref{dengjia4}) and Lemma \ref{yinli}, we finally obtain
\begin{align*}
\frac{d}{dt}E_{q}(t) + c_{0}E_{q}(t) \leq & \, C \alpha_{q}(1+t)^{-n/4} M(t)f(t) + C \alpha_{0}\|G_{1}, G_{3}\|_{B^{n/2-1,n/2}}   \\
& \, + C \alpha_{q} \|G_{0}, G_{2}\|_{\dot{B}^{n/2-1}}.
\end{align*}
\end{proof}

\section{Derive optimal time decay rate}

With the analysis about low and high frequency part, we now give the proof of Theorem \ref{MainTheorem}.
From Proposition \ref{HighMainPro}, we know that
\begin{align}\label{Decay1}
\begin{split}
E_{q}(t) \leq & e^{-c_{0}t} E_{q}(0) + C \int_{0}^{t} e^{-c_{0}(t-\tau)} \Big( \alpha_{q}(1+\tau)^{-n/4}M(\tau)f(\tau)  \\
& + \alpha_{q}\|G_{1}, G_{3}\|_{B^{n/2-1,n/2}} + \alpha_{q} \|G_{0}, G_{2}\|_{\dot{B}^{n/2-1}} \Big)d\tau.
\end{split}
\end{align}
Through homogeneous para-differential calculus, we can get
\begin{align}\label{DG1}
\begin{split}
\|G_{1}\|_{B^{n/2-1,/2}} \leq & \, C\, \|a\|_{B^{n/2-1,n/2}} \|\mathrm{div}v\|_{\dot{B}^{n/2}}    \\
\leq & \, C\, (1+\tau)^{-n/4} M(\tau) f(\tau),
\end{split}
\end{align}
\begin{align}\label{DG3}
\begin{split}
\|G_{3}\|_{B^{n/2-1,n/2}} \leq & \, C\, \|F\|_{B^{n/2-1,n/2}} \|\nabla v\|_{\dot{B}^{n/2}}    \\
\leq & \, C\, (1+\tau)^{-n/4}M(\tau)f(\tau),
\end{split}
\end{align}
\begin{align}\label{DG0}
\begin{split}
\|G_{0}\|_{\dot{B}^{n/2-1}} \leq & \, C\, \|a F\|_{\dot{B}^{n/2}} \leq \, C\, \|a\|_{\dot{B}^{n/2}} \|F\|_{\dot{B}^{n/2}}   \\
\leq & \, C\, (1+\tau)^{-n/2} M^{2}(\tau).
\end{split}
\end{align}
For the term $G_{2}$, we need to estimates term by term carefully as follows
\begin{align}\label{DG21}
\|v\cdot\nabla e\|_{\dot{B}^{n/2-1}} + \|v\cdot\nabla v\|_{\dot{B}^{n/2-1}} \leq & \, C\, \|v\|_{\dot{B}^{n/2-1}} \|\nabla v\|_{\dot{B}^{n/2}}    \\
\leq & \, C(1+\tau)^{-n/4}M(\tau)f(\tau).
\end{align}
Noting that $C(0) = K(0) = 0$, we get by using Lemma 3 and Remark 6 in \cite{L2viscoelastic} that
\begin{align}\label{DG22}
\|C(a)\mathcal{A}v\|_{\dot{B}^{n/2-1}} \leq \, C\|\nabla^{2}v\|_{\dot{B}^{n/2-1}} \|C(a)\|_{\dot{B}^{n/2}} \leq \, C (1+\tau)^{-n/4}M(\tau)f(\tau),
\end{align}
\begin{align}\label{DG23}
\|K(a)\nabla a\|_{\dot{B}^{n/2-1}} \leq \, C\|K(a)\|_{\dot{B}^{n/2}} \|\nabla a\|_{\dot{B}^{n/2-1}} \leq \, C (1+\tau)^{-n/2}M^{2}(\tau),
\end{align}
\begin{align}\label{DG24}
\|F\nabla F\|_{\dot{B}^{n/2-1}} \leq \, C \|F\|_{\dot{B}^{n/2}} \|\nabla F\|_{\dot{B}^{n/2-1}} \leq \, C(1+\tau)^{-n/2}M^{2}(\tau).
\end{align}
From the above estimates (\ref{DG21})-(\ref{DG24}), we obtain
\begin{align}\label{DG2}
\|G_{2}\|_{\dot{B}^{n/2-1}} \leq \, C(1+\tau)^{-n/2}M^{2}(\tau) + C\, (1+\tau)^{-n/4}M(\tau)f(\tau).
\end{align}
Substitute (\ref{DG1})-(\ref{DG0}) and (\ref{DG2}) into (\ref{Decay1}), we will have
\begin{align*}
\sum_{q\geq R} E_{q}(t) \leq & e^{-c_{0}t}\sum_{q\geq R}E_{q}(0) + C\int_{0}^{t}e^{-c_{0}(t-\tau)}\Big( (1+\tau)^{-n/2}M^{2}(\tau)    \\
& + (1+\tau)^{-n/4}M(\tau)f(\tau) \Big) d\tau   \\
\leq & e^{-c_{0}t}\sum_{q\geq R}E_{q}(0) + M(t) \int_{0}^{t}e^{-c_{0}(t-\tau)}(1+\tau)^{-n/4}f(\tau)d\tau \\
& + M^{2}(t) \int_{0}^{t} e^{-c_{0}(t-\tau)}(1+\tau)^{-n/2} d\tau   \\
\leq & e^{-c_{0}t}\sum_{q\geq R}E_{q}(0) + C (1+t)^{-n/4}\delta M(t) + C(1+t)^{-n/2}M^{2}(t).
\end{align*}
So we obtain
\begin{align*}
(1+t)^{n/4}\sum_{q\geq R}E_{q}(\tau) \leq \, C\left( \|(a_{0}, F_{0})\|_{B^{n/2}} + \|v_{0}\|_{B^{n/2-1}} \right) + C\delta M(t) + CM^{2}(t).
\end{align*}
Combining the above inequality, Remark \ref{G1G2G3} and using properties of homogeneous Besov space, we obtain
\begin{align}\label{zuihou}
M(t) \leq \, C\left( \|(a_{0}, F_{0})\|_{B^{n/2}} + \|v_{0}\|_{B^{n/2-1}} \right) + C\delta M(t) + CM^{2}(t).
\end{align}
By taking $\delta > 0$ suitably small, we finally have
\begin{align}\label{zuihouhou}
M(t) \leq \, C\left( \|(a_{0}, F_{0})\|_{B^{n/2}} + \|v_{0}\|_{B^{n/2-1}} \right)
\end{align}
for all $0 \leq t\leq T$.
It follows from local well-posedness Theorem \ref{localeuler} and the above estimate (\ref{zuihouhou}) that
\begin{align*}
M(t) \leq C < \infty
\end{align*}
for all $t > 0$.
Hence, we obtain the desired decay estimates in Theorem \ref{MainTheorem}.

\section{Acknowledgements}

This research is support partially by National Natural Science Foundation of China under the grant no. 11131006, and by the National Basic Research Program of China under the grant no. 2013CB329404.
%The authors wish to thank the anonymous reviewers for their helpful
%and insightful comments and suggestions.


\begin{thebibliography}{1}

\bibitem{leizhen2008}Z. Lei, C. Liu, Y. Zhou, Global solutions of incompressible viscoelastic fluids, Arch. Rational Mech. Anal. 188 (2008) 371-398.
\bibitem{Linfanghua2005}F. Lin, C. Liu, P. Zhang, On hydrodynamics of viscoelastic fluids, Comm. Pure Appl. Math. 58 (2005) 1437-1471.
\bibitem{Linfanghua2008}F. Lin, P. Zhang, On the initial-boundary value problem of the incompressible viscolastic fluid system, Commu. Pure Appl. Math. 61 (2008) 539-558.
\bibitem{LpIncompressibleViscoelastic}T. Zhang, D. Fang, Global Existence of Strong Solution for Equations Related to the Incompressible Viscoelastic
Fluids in the Critical $L^{p}$ Framework, SIAM J. Math. Anal. 44(2013) 2266-2288.
\bibitem{L2viscoelastic}J. Qian, Z. Zhang, Global Well-Posedness for Compressible Viscoelastic Fluids near Equilibrium, Arch.Rational
Mech. Anal. 198(2010) 835-868.
\bibitem{XianpengHu}X. Hu, D. Wang, Global existence for the multi-dimensional compressible viscoelastic flows,
J. Differential Equations 250(2011) 1200-1231.
\bibitem{Lions2000}P. Lions, N. Masmoudi, Global solutions for some Oldroyd models of non-Newtonian flows, Chin. Ann. Math. Ser. B. 21 (2000) 131-146.
\bibitem{Jia2014}J. Jia, J. Peng, Z. Mei, Well-posedness and time-decay for compressible viscoelastic fluids in critical Besov space, J. Math. Anal. Appl. 418 (2014) 638-675
\bibitem{ponce1985}G. Ponce, Global existence of small solution to a class of nonlinear evolution equations, Nonlinear Anal. 9 (1985) 339-418.
\bibitem{yanguo2012}Y. Guo, Y. Wang, Decay of dissipative equations and negative sobolev spaces, Comm. Partial Differ. Equ. 37 (2012) 2165-2208.
\bibitem{HailiangLi2009}H. Li, T. Zhang, Large time bahavior of isentropic compressible Navier-Stokes system in $\mathbb{R}^{3}$, Math. Meth. Appl. Sci. 34 (2011) 670-682.
\bibitem{hoff1995}D. Hoff, K. Zumbrun, Multi-dimensional diffusion waves for the Navier-Stokes equations of compressible flow, Indiana Univ. Math. J. 44 (1995) 604-676.
\bibitem{hoff1997}D. Hoff, K. Zumbrun, Pointwise decay estimates for multi-dimensional Navier-Stokes diffusion waves, Z. Angew. Math. Phys. 48 (1997) 597-614.
\bibitem{liutaiping1998}T. Liu, W. Wang, The pointwise estimates of diffusion waves for the Navier-Stokes equaions in odd multi-dimensions,
Comm. Math. Phys. 196 (1998) 145-173.
\bibitem{DLLi}D. L. Li, The green's function of the Navier-Stokes equations for gas dynamics in $\mathbb{R}^{3}$, Comm. Math. Phys. 257 (2005) 579-619.
\bibitem{Okita2014}M. Okita, Optimal decay rate for strong solutions in critical spaces to the compressible Navier-Stokes equations, J. Differential Equations 257 (2014) 3850-3867.
\bibitem{decayhu}X. Hu, G. Wu, Global existence and optimal decay rates for three-dimensional compressible viscoelastic flows. (English summary)
SIAM J. Math. Anal. 45 (2013), 2815-2833.
\bibitem{chemin}Chemin, J.-Y.; Lerner, N. Flot de champs de vecteurs non lipschitziens et$\acute{e}$quations de Navier-Stokes,
J. Differential Equations 121, 314-328 (1992).
\bibitem{FourierPDE}H. Bahouri, J. Chemin, R. Danchin, Fourier Analysis and Nonlinear Partial Differential Equations,
Grundlehren der mathematischen Wissenschaften, Springer, (2011).
\bibitem{Danchin2005}R. Danchin, Fourier Analysis Methods for PDE's, Lecture Notes (2005) November 14.
\bibitem{NavierStokesLp}Q. Chen, C. Miao, Z. Zhang, Global Well-Posedness for Compressible Navier-Stokes Equations with Highly Oscillating Initial
Velocity, Comm. Pure. Appl. Math. Vol. LXIII (2010) 1173-1224.
\bibitem{NavierStokesLpDanchin}F. Charve, R. Danchin, A Gloal Existence Result for the Compressible Navier-Stokes equations in
the critical $L^{p}$ framework, Arch. Ration. Mech. Anal. 198(2010) 233-271.
\bibitem{Hoff1995}D. Hoff, K. Zumbrun, Mutli-dimensional diffusion waves for the Navier-Stokes equations of compressible flow,
Indiana Univ. Math. J. 44 (1995) 603-676.


\end{thebibliography}
\end{document}